\documentclass{article}

\usepackage{arxiv}

\usepackage[utf8]{inputenc} 
\usepackage[T1]{fontenc}    
\usepackage{hyperref}       
\usepackage{url}            
\usepackage{booktabs}       
\usepackage{amsfonts}       
\usepackage{nicefrac}       
\usepackage{microtype}      
\usepackage{lipsum}

\usepackage{amsmath}

\usepackage{color}
\usepackage{amsthm}
\theoremstyle{definition}
\newtheorem{assumption}{Assumption}[section]
\newtheorem{definition}{Definition}[section]
\newtheorem{prop}{Proposition}[section]
\newtheorem{lemma}{Lemma}[section]
\newtheorem{thm}{Theorem}[section]

\newtheorem{corr}{Corollary}[section]
\theoremstyle{remark}

\usepackage{graphicx}
\usepackage{subcaption}
\title{The random periodic solution of a stochastic differential equation with a monotone drift and its numerical approximation}

\author{
  Yue Wu \\
  Mathematical Institute\\
  University of Oxford\\
  Oxford,  OX2 6GG, UK\\ 
  Alan Turning Institute\\
  London, UK\\
  \texttt{yue.wu@maths.ox.ac.uk}  \\
}

\begin{document}
\maketitle

\begin{abstract}
In this paper we study the existence and uniqueness of the random periodic solution for a stochastic differential equation with an one-sided Lipschitz condition (also known as monotonicity condition) and the convergence of its numerical approximation via the backward Euler-Maruyama method. The existence of the random periodic
solution is shown as the limits of the pull-back flows of the SDE and the discretized SDE respectively. We establish a convergence rate of the strong error for the backward Euler-Maruyama method and obtain the weak convergence result for the approximation of the periodic measure.
\end{abstract}


\section{Introduction}
Periodicity is widely exhibited in a large number of natural phenomena like oscillations, waves, or even lying behind many complicated ensembles such as biological and economic systems. However, periodic behaviours are often found to be subject to random perturbation or under the influence of noise. Physicists have attempted to study random perturbations to periodic solutions for some time by considering a first linear approximation or asymptotic expansions in small noise regime, but this approach restricted its applicability to the small fluctuation (c.f. Van Kampen \cite{van2007}, Weiss and Knoblock \cite{weiss2009}). It was only until recently that the random periodic solution was endowed with a proper definition (c.f. Zhao and Zheng \cite{zhao2009}, Feng, Zhao and Zhou \cite{rpssde2011}), which is compatible with definitions of both the stationary solution (also termed as random fixed points) and the deterministic periodic solution. It gives a rigorous and clearer understanding to physically interesting problems of certain random phenomena with a periodic nature, and also represents a long time limit of the underlying random dynamical system. 

Let us recall the definition of the random periodic solution for stochastic semi-flows given in \cite{rpssde2011}.
Let $H$ be a separable Banach space. Denote by $(\Omega,{\cal F},\mathbb{P},(\theta_s)_{s\in \mathbb{R}})$ a metric dynamical system and $\theta_s:\Omega\to\Omega$ is  assumed to be measurably invertible for all $s\in \mathbb{R}$.
Denote $\Delta:=\{(t,s)\in \mathbb{R}^2, s\leq t\}$. Consider a stochastic semi-flow $u: \Delta \times \Omega\times H\to H$, which satisfies the following standard condition
\begin{eqnarray}{\label{16}}
u(t,r,\omega)=u(t,s,\omega)\circ u(s,r,\omega),\ \ {\rm for\ all } \ r\leq s\leq t,\ r, s,t\in \mathbb{R},\ \mbox{for }a.e.\ \omega\in\Omega.
\end{eqnarray}
We do not assume the map $u(t,s,\omega): H\to H$ to be invertible for $(t,s)\in \Delta,\ \omega \in \Omega$. 

\begin{definition}
\label{feng-zhao1}
A {\bf random periodic path of period $\tau$ of the semi-flow} $u: \Delta\times \Omega\times H\to H$
 is an ${\cal F}$-measurable
map $y:\mathbb{R}\times \Omega\to H$ such that
\begin{eqnarray}
\Big\{\begin{array}{l}u(t,s, y(s,\omega), \omega)=y(t,\omega),\ \ \forall t\geq s\\

y(s+\tau,\omega)=y(s, \theta_\tau \omega),\ \ \forall s\in  \mathbb{R}
                                     \end{array}
\end{eqnarray}
for any $\omega\in \Omega$.
\end{definition}

Building on this new concept, there have been more recent progresses towards understanding the random periodicity of various stochastic systems. The existence of random periodic solutions to stochastic differential equations (SDEs) and stochastic partial differential equations (SPDEs) are initially studied in \cite{rpssde2011} and \cite{rpsspde2011}, with additive noise. Instead of following the traditional geometric method of establishing the Poincar\'e mapping, a new analytical method for coupled infinite horizon forward-backward integral equations is introduced. It was then followed by the study on the anticipating random periodic solutions (c.f. Feng, Wu and Zhao: \cite{wu2016} and \cite{wu2018}) and the random periodicity of the stochastic functional differential equations (c.f. Feng, Luo and Zhao \cite{luo2014}). Regarding applications, Chekroun, Simonnet and Ghil \cite{chekroun2011} employed random periodic results to climate dynamics, and Wang \cite{wang2014} observed random peridicity behaviour in the study of birfurcations of stochastic reaction diffusion equations. 

An alternative approach to understand random periodic behaviours of SDEs is to study periodic measures which describe periodicity in the sense of distributions (c.f. Has'minskii \cite{minskii1980}). There are a few works in the literature attempting to study statistical solutions of certain types of SDEs with periodic forcings. This was motivated in the context of studying the climate change problem when the seasonal cycle is taken into considerations (c.f. Gershgorin and Majda \cite{gershgorin2010}, Majda and Wang \cite{majda2010}), in the context of the Brusselator arising in chemical reactions and Ornstein-Uhlenbeck processes (c.f. Scheutzow \cite{scheutzow1986}). It's worth noticing that random periodic solutions give rise to periodic measures (c.f. Feng and Zhao \cite{feng2020}), which is defined as follows.
\begin{definition}\label{def:measure}
Let $\mathcal{P}(\mathbb{R}^d)$ denote all probability measures on $\mathbb{R}^d$. The measure function $\rho_{\cdot}:\mathbb{R}\to \mathcal{P}(\mathbb{R}^d)$ is called a \textbf{periodic measure} if it satisfies for any
$s \in \mathbb{R}$, $t \geq 0$, and $\Gamma\in \mathcal{B}(\mathbb{R}^d)$,
\begin{equation}
    \rho_{s+\tau}=\rho_{s},\ \ \int_{\mathbb{R}^d}P(t + s, s, x, \Gamma)\rho_s(\mathrm{d}x)=\rho_{t+s}(\Gamma),
\end{equation}
where the transition probability of the semi-flow $u$ is set to be $P(t + s, s, \xi, \Gamma):=P(\{\omega:u(t+s,s,\omega)\xi\in \Gamma\})$.
\end{definition}
Conversely, from a periodic measure one can construct an enlarged probability space and random periodic process whose law is the periodic measure. It was then proved that the strong law of large numbers (SLLN) 
holds for periodic measures and corresponding random periodic processes. 

In general, random periodic solutions cannot be solved explicitly. One may treat the numerical approximation that stay sufficient close to the true solution as a good substitute to study stochastic dynamics. It is worth mentioning here that this is a numerical approximation of an
infinite time horizon problem. The classical numerical approaches including the Euler-Marymaya method and a modified Milstein method to simulate random period solutions of a dissipative system with global Lipchitz condition have been investigated in \cite{rpsnumerics2017}, which is the first paper that numerical schemes were used to approximate the random period trajectory. 

In this paper, we study the random periodic solutions of stochastic differential equations with weakened conditions on the drift term compared to \cite{rpsnumerics2017}, and simulate them via the backward Euler-Maruyama method. Here $H=\mathbb{R}^d$. Let $W \colon \mathbb{R} \times \Omega \to \mathbb{R}^d$ be a
standard two-sided Wiener process on the probability space $(\Omega, \mathcal{F}, \mathbb{P})$, with the filtration defined by $\mathcal{F}_s^{t}:=\sigma\{W_u-W_v:s<v\le u<t\}$ and $\mathcal{F}^{t}=\mathcal{F}^{t}_\infty=\vee_{s\le t}\mathcal{F}^{t}_s$. Throughout this paper, we shall use $|\cdot|$ for the Euclidean norm, $\|u\|:=\sqrt{\mathbb{E}[|u|^2]}$ and $\|u\|_p:=\sqrt[p]{\mathbb{E}[|u|^p]}$. We are interested in the $\mathbb{R}^d$-valued random periodic solution to a SDE of the form
\begin{align}
  \label{eq:SPDE}
  \begin{cases}
    \mathrm{d}{X^{t_0}_t} = 
    \big[-A X^{t_0}_t + f(t,X^{t_0}_t) \big] \mathrm{d}{t}+g(t)\mathrm{d}{W_t},& \quad \text{for } t \in (t_0,T],\\
    X^{t_0}_{t_0} = \xi,& 
  \end{cases}
\end{align}
where $\xi$ is a $\mathcal{F}^{t_0}$-measurable random initial condition. In addition, $A$, $f$, and $g$, and $\xi$ satisfy the following assumptions:
\begin{assumption}
  \label{as:A}
  The linear operator $A \colon \mathbb{R}^d \to \mathbb{R}^d$ is densely defined,
  self-adjoint, and positive definite with compact inverse.
\end{assumption}

Assumption \ref{as:A} implies the existence of a positive, increasing
sequence $(\lambda_i)_{i\in \mathbb{N}} \subset \mathbb{R}$ such that 
$0<\lambda_1 \le \lambda_2 \le \ldots \lambda_d$, and of an orthonormal basis $(e_i)_{i\in
\mathbb{N}}$ of $ \mathbb{R}^d$ such that $A e_i = \lambda_i  e_i$ for every $i \in
[d]$, where $[d]:=\{1,\ldots,d\}$.

\begin{assumption}
  \label{as:f}
  The mapping $f \colon \mathbb{R} \times \mathbb{R}^d \to \mathbb{R}^d$ is continuous and periodic in time with period $\tau$.
  Moreover, there exists a $C_f \in
  (0,\infty)$ such that 
  \begin{align*}
   & \langle u_1-u_2, f(t,u_1) - f(t,u_2)  \rangle \le C_f |u_1 - u_2|^2\\
    &  \langle u, f(t,u) \rangle \le C_f (1+|u|^2)
  \end{align*}
  for all $u,u_1, u_2 \in \mathbb{R}^d$ and $t \in [0,\tau)$.
\end{assumption}

\begin{assumption}
  \label{as:g}
  The diffusion coefficient functions $g \colon  \mathbb{R} \to \mathbb{R}$ is continuous and periodic in time with period $\tau$. Moreover, we assume there exists a constant $\sigma >0$ such that $\sup_{s\in [0,\tau)}|g(s)|<\sigma$ and $|g(t_1)-g(t_2)|\leq \sigma |t_2-t_1|$ for all $t_1,t_2 \in [0,\tau)$.
\end{assumption}

It is well known that under these assumptions the \emph{ solution} $X_{\cdot}^{t_0} \colon [t_0,T] \times
\Omega \to \mathbb{R}^d$ to \eqref{eq:SPDE} is uniquely determined by the
variation-of-constants formula 
\begin{align}
  \label{eq:mild}
  X^{t_0}_t(\xi) = e^{-A(t-t_0)} \xi + \int_{t_0}^t e^{-A(t-s)} f(s,X^{t_0}_s) \mathrm{d}{s} + \int_{t_0}^t e^{-A(t-s)}g(s)
  \mathrm{d}{W_s}.
\end{align}

\subsection{The pull-back}
  We know there exists a standard $\mathbb{P}$-preserving ergodic Wiener shift $\theta$ such that $\theta_t (\omega)(s)=W_{t+s}-W_s$ for $s,t\in \mathbb{R}$, ie, $$\mathbb{P}\circ (\theta_t W_s)^{-1}=\mathbb{P}\circ (W_{t+s}-W_s)^{-1}.$$
 Due to being non-autonomous, $X$ does not satisfy the cocycle property \cite{arnold1995}. But we are able to verify that the $u(t,t_0): \Omega\times\mathbb{R}^d \to  \mathbb{R}^d$ given by $u(t,t_0)\xi=X_{t}^{t_0}(\xi)$ satisfies the semi-flow property and periodic property in Definition \ref{feng-zhao1}. Denote by $X^{-k\tau}_t(\xi,\omega)
$ the solution starting from time $-k\tau$. We will show that when $k\to \infty$, the
pull-back $X^{-k\tau}_t(\xi)$ has a limit $X^*
_t$ in $L^2(\Omega)$ and $X^*
_t$ is the random periodic solution of SDE \eqref{eq:SPDE}, satisfying
\begin{equation}
  \label{eq:limit}
  X^*_t =  \int_{-\infty}^t e^{-A(t - s)} f(s,X^{*}_s) \mathrm{d}{s} + \int_{-\infty}^t e^{-A(t - s)} g(s)
  \mathrm{d}{W_s}.
\end{equation}
To achieve it, we need additional assumptions on $\xi$ and $f$.
\begin{assumption}
  \label{as:f_constant}
  $C_f<\lambda_1$.
\end{assumption}
\begin{assumption}
  \label{as:ini}
  There exists a constant $C_\xi$ such that $\|\xi\|<C_\xi$.
\end{assumption}
\begin{assumption}
  \label{as:f_tangent}
  There exists a constant $\hat{C}_f$ such that $\Big|f(t,u)-\frac{\langle f(t,u),u\rangle}{|u|^2}u\Big|\leq \hat{C}_f(1+|u|)$ for $u\in\mathbb{R}^d, t\in[0,\tau)$.
\end{assumption}

Assumption \ref{as:f_tangent} together with Assmption \ref{as:A} to \ref{as:g} ensures the existence of a global semiflow generated from SDE \eqref{eq:SPDE} with additive noise \cite{scheutzow2017}. Section \ref{sec:existence} is devoted to the first main result, which claims the existence and uniqueness of random periodic solutions to the SDE \eqref{eq:SPDE} under the one-sided Lipchitz condition on the drift.
\begin{thm}\label{thm:main1} Under Assumption \ref{as:A} to Assumption \ref{as:f_tangent}, there exists a unique random periodic solution $X^{*}(r,\cdot)\in L^2(\Omega)$ such that the solution of \eqref{eq:SPDE} satisfies
\begin{equation}\label{eqn:lim_sde}
    \lim_{k\to \infty}\|X^{-k\tau}_t(\xi)-X^*_t\|=0.
\end{equation}
\end{thm}
In Section \ref{sec:sol}, we derive additional properties of the solution such as the uniform boundedness for a higher moment of $X^{-k\tau}_t$ and solution regularity under an additional Assumption \ref{as:f_Khasminskii-type}, which imposes superlinearity of $f$ and assumes a larger lowerbound for $\lambda_1$ compared to Assumption \ref{as:f_constant}. Those properties will play an important role in proving the order of convergence of the backward Euler-Maruyama in Theorem \ref{thm:error}.

\subsection{The backward Euler-Maruyama}
For stiff ordinary differential equations, the implicit method is preferred due to its good performance even on a time grid with a large step size \cite{wanner1996}. For its stochastic counterpart such as \eqref{eq:SPDE}, we shall approximate the solution using the backward Euler-Maruyama method, the simplest version of implicit methods for SDEs.

Let us fix an equidistant partition $\mathcal{T}^h:=\{jh,\ j\in \mathbb{Z} \}$ with stepsize $h\in (0,1)$. Note that $\mathcal{T}^h$ stretches along the real line because eventually we are dealing with an infinite time horizon problem with the form of \eqref{eq:limit}. Then to simulate the solution to \eqref{eq:SPDE} starting at $-k\tau$, the \emph{backward Euler-Maruyama method} on $\mathcal{T}^h$ is given by the recursion 
\begin{align}
  \label{eq:RandM}
  \begin{split}
    \hat{X}_{-k\tau+(j+1)h}^{-k\tau} =& \hat{X}_{-k\tau+jh}^{-k\tau} - Ah\hat{X}_{-k\tau+(j+1)h}^{-k\tau}+h f\big((j+1)h,   \hat{X}_{-k\tau+(j+1)h}^{-k\tau} \big)+ 
    g(jh)\Delta W_{-k\tau+jh}
  \end{split}
\end{align}
for all $j \in \mathbb{N}$, where the initial value $\hat{X}_{-k\tau}^{-k\tau}  = \xi$, and $\Delta W_{-k\tau+jh}:=W_{-k\tau+(j+1)h}-W_{-k\tau+jh}$. Note that due to the periodicity of $f$ (c.f. Assumption \ref{as:f}), we write $f(-k\tau+jh, \hat{X}_{-k\tau+jh}^{-k\tau} )$ as $f(jh, \hat{X}_{-k\tau+jh}^{-k\tau})$, and similar arguments for the $g$ term. 

The implementation of \eqref{eq:RandM} requires solving a nonlinear equation at each iteration. Theorem \ref{thm:well-posedness} ensures the well-poseness of difference equation \eqref{eq:RandM} under Assumption \ref{as:A} to \ref{as:f_constant}. We explore the random periodicity of its solution in Section \ref{sec: num1} and prove the second main result in our paper:
\begin{thm}\label{thm:main2}
Under Assumption \ref{as:A} to Assumption \ref{as:ini}, for any $h\in(0,1)$ with $\tau=nh$, $n\in\mathbb{N}$, the backward Euler-Maruyama method \eqref{eq:RandM} admits a random period solution on $\mathcal{T}^h$. 
\end{thm}
We also determine a strong order $\frac{1}{2}$ for the backward Euler-Maruyama method in Theorem \ref{thm:error} and Corrolary \ref{corr:error}. Compared to Theorem 3.4 and Theorem 4.2 in \cite{rpsnumerics2017} which imposed condition on the size of $h$ (to be sufficient small) because of the implementation of explicit numerical methods, we benefit a flexible choice of stepsize $h$ from using the backward Euler-Maruyama method even in the infinite horizon case.

In Section \ref{sec:periodic} we consider the convergence of transition
probabilities generated by Eqn. \eqref{eq:SPDE} and its numerical scheme to the
periodic measure and discretised periodic measure, respectively, and error estimate of the two periodic
measures in the weak topology.

Finally we assess the performance of the backward Euler-Maruyama method via a numerical experiment and compare it with the one of the classical Euler-Maruyama method under various steps. The result shows that the backward Euler-Maruyama method is able to converge to the random periodic solution when the stepsize is fairly large while Euler-Maruyama method diverges. 
\section{Preliminaries}
In this section we present a few useful mathematical tools for later use.
\begin{thm}[The Gr\"onwall inequality: a continuous version]
Let $I$ denote a time interval in form of $[I_-,I^+]$. Let $a$, $b$ and $u$ be real-valued functions defined on $I$. Assume that $b$ and $f$ are continuous and that the negative part of $a$ is integrable on every closed and bounded subinterval of $I$. Then
if $b$ is nonnegative and if $u$ satisfy the following inequality
\begin{equation}\label{eqn:gronwall}
  u(t)\leq a(t)+\int_{I_-}^{t}b(s)u(s)\mathrm{d}s,
\end{equation}
then
\begin{equation}\label{eqn:gronwall2}
    u(t)\leq a(t)+\int_{I_-}^t a(s)b(s)\exp{\Big(\int_s^t b(r)\mathrm{d}r\Big)}\mathrm{d}s.
\end{equation}
If in addition, the function $a$ is non-decreasing, then
\begin{equation}\label{eqn:gronwall3}
    u(t)\leq a(t)\exp{\Big(\int_{I_-}^t b(r)\mathrm{d}r\Big)}.
\end{equation}
\end{thm}
\begin{thm}[The Gr\"onwall inequality: a discrete version \cite{willett1965,yevik2011}]
Consider two nonnegative sequences $(u_n)_{n\in \mathbb{N}},
  (a_n)_{n\in \mathbb{N}} \subset \mathbb{R}$ which for some given $w \in [0,\infty)$ satisfy 
  $$u_n \leq a_n + w\sum_{j=1}^{n-1}  u_j,\quad \text{ for all }  n \in \mathbb{N}.$$
  Then, for all $n \in \mathbb{N}$, it also holds true that
  $$u_n \leq a +\frac{w}{c_{n-1}}\big(u_0+\sum_{j=1}^{n-1} a_jc_j \big),$$
  where $c_j:=\frac{1}{(1+w)^j}$ for $j\in\mathbb{N}$.
\end{thm}

Also the crucial equality for analysis of the backward Euler-Maruyama is 
\begin{equation}\label{eqn:eqa}
    |b|^2-|a|^2+|b-a|^2=2\langle b-a, b\rangle.
\end{equation}
\section{Existence and uniqueness of the random periodic solution}\label{sec:existence}
We focus on the existence and uniqueness of the random periodic solution to SDE \eqref{eq:SPDE} in this section. To achieve it, we first show there is a uniform bound for the second moment of its solution under necessary assumptions.

\begin{lemma}\label{lem:boundedness1}
For SDE \eqref{eq:SPDE} with given initial condition $\xi$ and satisfying Assumption \ref{as:A} to \ref{as:ini}, we have
\begin{equation} 
    \sup_{k\in \mathbb{N}}\sup_{t>-k\tau}\mathbb{E}[|X_{t}^{-k\tau}(\xi)|^2]\leq C_{\xi}^2+\frac{2K_2\lambda_1}{2(\lambda_1-C_f)},
\end{equation}
where $K_2:=\frac{\sigma^2+2C_f}{2\lambda_1}$.
\end{lemma}
\begin{proof} 
Applying It\^o formula to $e^{2\lambda_1 t}\|X_{t}^{-k\tau}(\xi)\|^2$ and taking the expectation yield
\begin{align}
\begin{split}
     &e^{2\lambda_1 t}\mathbb{E}[|X_{t}^{-k\tau}(\xi)|^2]= e^{-2\lambda_1 k\tau}\mathbb{E}[|\xi|^2]+2\lambda_1 \int_{-k\tau}^t e^{2\lambda_1 s}\mathbb{E}[|X_{s}^{-k\tau}|^2]\mathrm{d}s\\
 &-2\int_{-k\tau}^te^{2\lambda_1 s}\mathbb{E}\langle X_{s}^{-k\tau}, AX_{s}^{-k\tau}\rangle\mathrm{d}s+2\int_{-k\tau}^te^{2\lambda_1 s}\mathbb{E}\langle X_{s}^{-k\tau}, f(s,X_{s}^{-k\tau})\rangle \mathrm{d}s+\int_{-k\tau}^t e^{2\lambda_1 s}|g(s)|^2\mathrm{d}s.
\end{split}
\end{align}
Note that $2(\lambda_1I-A)$ is non-positive definite. Then making use of assumptions \ref{as:f} and \ref{as:g} gives
\begin{align*}
         &e^{2\lambda_1 t}\|X_{t}^{-k\tau}(\xi)\|^2\leq e^{-2\lambda_1 k\tau}\|\xi\|^2+2C_f\int_{-k\tau}^te^{2\lambda_1 s}\|X_{s}^{-k\tau}\|^2\mathrm{d}s+(\sigma^2+2C_f)\int_{-k\tau}^t e^{2\lambda_1  s}\mathrm{d}s\\
 &\le e^{-2\lambda_1 k\tau}\|\xi\|^2+\frac{(\sigma^2+2C_f)}{2\lambda_1}(e^{2\lambda_1 t}-e^{-2\lambda_1 k\tau})+2C_f\int_{-k\tau}^te^{2\lambda_1 s}\|X_{s}^{-k\tau}\|^2\mathrm{d}s.
\end{align*}
Denote $K_1:= e^{-2\lambda_1 k\tau}\big(\|\xi\|^2-\frac{\sigma^2+2C_f}{2\lambda_1}\big)$ and $K_3:=2C_f$. Note that  $K_3\leq 2\lambda_1$ because of assumption \ref{as:f_constant}. By the Gr\"onwall inequality, we have that
\begin{align*}
   & e^{2\lambda_1 t}\|X_{t}^{-k\tau}(\xi)\|^2\le K_1+K_2 e^{2\lambda_1 t}+\int_{-k\tau}^t( K_1+K_2 e^{2\lambda_1 s})K_3e^{K_3(t-s)}\mathrm{d}s\\
   &\leq K_1e^{K_3(k\tau+t)}+K_2 e^{2\lambda_1 t}+\frac{K_2K_3}{2\lambda_1-K_3}(e^{2\lambda_1 t}-e^{-2\lambda_1 k\tau})\\
   &\leq  (K_1e^{2\lambda_1k\tau}+K_2) e^{2\lambda_1 t}+\frac{K_2K_3}{2\lambda_1-K_3}e^{2\lambda_1 t}.
\end{align*}
Note that $K_1e^{2\lambda_1 k\tau}+K_2=\|\xi\|^2$. By Assumption \ref{as:ini} it leads to
$$\|X_{t}^{-k\tau}(\xi)\|^2\le\|\xi\|^2+\frac{K_2K_3}{2\lambda_1-K_3}\leq C_{\xi}^2+\frac{2K_2\lambda_1}{2\lambda_1-K_3}.$$

\end{proof}
Then we explore the solution dependence on initial conditions.
\begin{lemma}\label{lem:stable1}
Let Assumption \ref{as:A} to \ref{as:g} hold. Denote by $X_t^{-k\tau}$ and $Y_t^{-k\tau}$ two solutions of SDE \eqref{eq:SPDE} with different initial values $\xi$ and $\eta$. Then 
$$\|X_t^{-k\tau}-Y_t^{-k\tau}\|^2\leq e^{(C_f-\lambda_1)(t+k\tau)}\|\xi-\eta\|^2.$$
In addition, if Assumption \ref{as:f_constant} holds, then for every $\epsilon>0$, there exists a $t\geq -k\tau$ such that it holds  
\begin{equation}\label{eqn:stable1}
    \|X_{\tilde{t}}^{-k\tau}-Y_{\tilde{t}}^{-k\tau}\|^2<\epsilon
\end{equation}

whenever $\tilde{t}\geq t$.
\end{lemma}
\begin{proof}
Define $E_t^{-k\tau}:=X_t^{-k\tau}-Y_t^{-k\tau}$. From \eqref{eq:mild}, we have that
\begin{align}
    \begin{split}
      &E_t^{-k\tau} = (\xi-\eta)+\int_{-k\tau}^{t} e^{-A(t - s)} \big(f(s,X^{-k\tau}_s)- f(s,Y^{-k\tau}_s) \big)\mathrm{d}{s}.
    \end{split}
\end{align}
Similar as the proof of Lemma \ref{lem:boundedness}, we apply It\^o formula to $e^{2\lambda_1 t}|E_t^{-k\tau}|^2$, take the expectation, make use of Assumption \ref{as:f} and get
\begin{align}
    \begin{split}
     &e^{2\lambda_1 t}\|E_t^{-k\tau}\|^2\leq e^{-2\lambda_1 k\tau}\|\xi-\eta\|^2+2\int_{-k\tau}^te^{2\lambda_1 s}\mathbb{E}\Big\langle E_s^{-k\tau}, f(s,X_{s}^{-k\tau})-f(s,Y_{s}^{-k\tau})\Big\rangle\mathrm{d}s\\
     &\leq  e^{-2\lambda_1 k\tau}\|\xi-\eta\|^2+2C_f\int_{-k\tau}^te^{2\lambda_1 s}\|E_{s}^{-k\tau}\|^2\mathrm{d}s.
\end{split}
\end{align}
Applying Eqn. \eqref{eqn:gronwall3} gives the desired inequality. The claim in \eqref{eqn:stable1} follows if Assumption \ref{as:f_constant} holds.
\end{proof}
With Lemma \ref{lem:boundedness1}, Lemma \ref{lem:stable1} and Assumption \ref{as:f_tangent}, the main result Theorem \ref{thm:main1} can be shown by following the same argument in the proof of Theorem 2.4 in \cite{rpsnumerics2017}.

\section{More results on the solution}\label{sec:sol}
In this section, we mainly explore properties of the solution to \ref{eq:SPDE} for analysis later. 
\begin{assumption}
  \label{as:f_Khasminskii-type}
  There exists a constant $q\in (1,\infty)$ and a positive $L$ such that 
  $$|f(t_1,u_1)-f(t_2,u_2)|\leq L(1+|u_1|^{q-1}+|u_2|^{q-1})|u_1-u_2|,$$
  for $t_1,t_2\in [0,\tau)$ and $u_1,u_2\in \mathbb{R}^d$. In addition, there exists a positive number $p\in [4q-2,\infty)$ such that $$\gamma_p:=\big(C_f+\frac{(p-1)\sigma^2}{2}\Big)(2+p+2^{p+1})<p\lambda_1.$$
  
\end{assumption}
The first property we will show is the uniform boundedness for the $p$-th moment of the SDE solution.
\begin{prop} Under Assumption \ref{as:A} to \ref{as:ini} and Assumption \ref{as:f_Khasminskii-type}, the solution to \eqref{eq:SPDE} satisfies
\begin{equation} 
    \sup_{k\in \mathbb{N}}\sup_{t>-k\tau}\mathbb{E}[|X_{t}^{-k\tau}(\xi)|^p_p]<\infty.
\end{equation}
\end{prop}
\begin{proof}
From the proof of Lemma \ref{lem:boundedness1}, we know that
\begin{align*}
            \mathrm{d}e^{2\lambda_1 t}|X_{t}^{-k\tau}|^2&= 2\lambda_1 \mathrm{d}e^{2\lambda_1 t}|X_{t}^{-k\tau}|^2-2e^{2\lambda_1 t}\langle X_{t}^{-k\tau}, AX_{t}^{-k\tau}\rangle\mathrm{d}t\\
 &+2e^{2\lambda_1 t}\langle X_{t}^{-k\tau}, f(t,X_{t}^{-k\tau})\rangle \mathrm{d}t+2e^{2\lambda_1 t}|g(t)|^2\mathrm{d}t+2e^{2\lambda_1 t}\langle X_{t}^{-k\tau}, g(t)\rangle \mathrm{d}W_t.    
\end{align*}
Then applying It\^o formula to $e^{p\lambda_1 t}|X_{t}^{-k\tau}|^p=\big(e^{2\lambda_1 t}|X_{t}^{-k\tau}|^2\big)^{p/2}$ and taking into consideration $2(\lambda_1I-A)$ being non-positive definite give
\begin{align*}
    \mathbb{E}[e^{p\lambda_1 t}|X_{t}^{-k\tau}|^p]&\leq e^{p\lambda_1 t}\|\xi\|^p_p+p \int_{-k\tau}^{t}\mathbb{E}\Big[e^{(p-2)\lambda_1 s}|X_{s}^{-k\tau}|^{p-2}e^{2\lambda_1 s}\langle X_{s}^{-k\tau}, f(t,X_{s}^{-k\tau})\rangle\Big]\mathrm{d}s\\
    &+\frac{p(p-1)}{2} \int_{-k\tau}^{t}\mathbb{E}\Big[e^{(p-2)\lambda_1 s}|X_{s}^{-k\tau}|^{p-2}e^{2\lambda_1 s}\Big]g(s)^2\mathrm{d}s.
\end{align*}
Now by the Young inequality
$$a^{p-2}b\leq \frac{p-2}{p}a^p+\frac{2}{p}b^{p/2}, \forall a,b\geq 0,$$
and the inequality from fundamental calculus,
$$\big(a^2+b^2\big)^{\frac{p}{2}}\leq 2^p(a^p+b^p), \forall a,b\geq 0,$$
we have that
\begin{align*}
     \mathbb{E}[e^{p\lambda_1 t}|X_{t}^{-k\tau}|^p]&\leq e^{p\lambda_1 t}\|\xi\|^p_p+p\Big(C_f+\frac{(p-1)\sigma^2}{2}\Big) \int_{-k\tau}^{t}e^{p\lambda_1 s}\mathbb{E}\Big[|X_{s}^{-k\tau}|^{p-2}\big(1+|X_{s}^{-k\tau}|^{2}\big)\Big]\mathrm{d}s\\
     &\le e^{p\lambda_1 t}\|\xi\|^p_p+\gamma_p \int_{-k\tau}^{t}e^{p\lambda_1 s}\big(1+\|X_{s}^{-k\tau}\|^{p}_p\big)\mathrm{d}s\\
     & \leq \hat{K}_1+\gamma_p e^{p\lambda t}+\gamma_p \int_{-k\tau}^{t}\mathbb{E}[e^{p\lambda_1 s}|X_{s}^{-k\tau}|^{p}]\mathrm{d}s,
\end{align*}
where $\hat{K}_1:= e^{-p\lambda_1 k\tau}\big(\|\xi\|^p_p-\gamma_p\big)$. Because of Assumption \ref{as:f_Khasminskii-type}, the rest simply follows the same way as the end of the proof for Lemma \ref{lem:boundedness1}.
\end{proof}
Following a similar argument as in Proposition 5.4 and 5.5 \cite{kruse2016}, we can easily get the following bounds for analysis later.
\begin{prop}\label{prop:use_bound}Let Assumption \ref{as:A} to \ref{as:ini} and Assumption \ref{as:f_Khasminskii-type} hold. Then there exists a positive constant $C_{q,A,f}$ which depends on $q$, $d$, $A$,$C_f$ only, such that
\begin{align}
    \|X_{t_1}^{-k\tau}-X_{t_2}^{-k\tau}\| \leq C_{q,A,f}\big(1+\sup_{k\in\mathbb{N}}\sup_{t\geq -k\tau}\|X_{t}^{-k\tau}\|^q_{2q}\big)|t_2-t_1|^{\frac{1}{2}},
\end{align}
for all $t_1,t_2\geq -k\tau$. Moreover, 
\begin{align}
    \begin{split}
        &\int_{t_1}^{t_2}\big\|A\big(X_{s}^{-k\tau}-X_{t_4}^{-k\tau}\big)+f\big(s,X_{s}^{-k\tau}\big)-f\big(t_3,X_{t_4}^{-k\tau}\big)\big)\big\|\mathrm{d}s\\
        &\leq C_{q,A,f}\big(1+\sup_{k\in\mathbb{N}}\sup_{t\geq -k\tau}\|X_{t}^{-k\tau}\|^{2q-1}_{4q-2}\big)|t_2-t_1|^{\frac{3}{2}},
    \end{split}
\end{align}
for all $t_3,t_4\in [t_1,t_2]$.
\end{prop}
\section{The random periodic solution of the backward Euler-Maruyama scheme}\label{sec: num1}
In this section we will prove that the backward Euler-Maruyama method \eqref{eq:RandM} admits a unique 
discretized random period solution. To achieve this, let us first show the existence and uniqueness of solution to the targeted scheme.

\begin{thm}[Well-posedness]
  \label{thm:well-posedness}
  Let Assumption~\ref{as:A} to \ref{as:f_constant} be satisfied. Then for any $h\in (0,1)$, there exists a unique $\mathbb{R}^d$-valued sequence
  $(\hat{X}^{-k\tau}_{jh})_{h\in \mathbb{N}}$ satisfying the difference equation \eqref{eq:RandM} on the associated time grid $\mathcal{T}^h$. 
\end{thm}
\begin{proof}
  Let $h \in (0,1)$ and define $G \colon \mathbb{R}^d \to \mathbb{R}^d$ by $G_(\zeta) =
  \zeta +Ah\zeta- h f(t, \zeta)$ for all $\zeta \in \mathbb{R}^d$ and $t\in[0,\tau)$. Then it holds
  \begin{align*}
    \langle G_t(\varsigma) - G_t(\zeta), \varsigma - \zeta \rangle
    = (I+Ah)|\varsigma - \zeta|^2 - h   
    \langle f(t,\varsigma) - f(t,\zeta), \varsigma - \zeta \rangle
    \ge (1 +\lambda_1h- C_f h) |\varsigma - \zeta|^2.
  \end{align*}
  Because of Assumption \ref{as:f_constant}, we have $L_{G_t} :=1 +\lambda_1h- C_f h>1$. Hence,
  the uniform monotonicity theorem \footnote{For a proof we refer to \cite{ortega2000,stuart1996}.} (c.f. Proposition 3.5 in \cite{wu2020}) is applicable. In particular,
  the sequence $(\hat{X}^{-k\tau}_{jh})_{h\in \mathbb{N}}$ defined by 
  \begin{align*}
    \hat{X}^{-k\tau}_{(j+1)h} := G_{(j+1)h}^{-1}\big(\hat{X}^{-k\tau}_{jh} + g(jh)\Delta W_{-k\tau+jh})
  \end{align*}
  for every $j \in \mathbb{N}$ satisfies \eqref{eq:RandM}.
\end{proof}
The next Lemma claims there is a uniform bound for the second moment of the numerical solution under necessary assumptions.
\begin{lemma}\label{lem:boundedness}
Under Assumption \ref{as:A} to \ref{as:ini}, for any $h\in (0,1)$, it holds for the backward Euler-Maruyama method \eqref{eq:RandM} on $\mathcal{T}^h$ that
\begin{equation} 
    \sup_{k,N\in\mathbb{N}}\mathbb{E}[|\hat{X}_{-k\tau+Nh}^{-k\tau} (\xi)|^2]< \infty.
\end{equation}
\end{lemma}
\begin{proof}
First note that from \eqref{eqn:eqa} we have that for any $N\in \mathbb{N}$
\begin{align}
\begin{split}
&|\hat{X}_{-k\tau+Nh}^{-k\tau}|^2-|\hat{X}_{-k\tau+(N-1)h}^{-k\tau}|^2+|\hat{X}_{-k\tau+Nh}^{-k\tau}-\hat{X}_{-k\tau+(N-1)h}^{-k\tau}|^2\\
&=2\langle \hat{X}_{-k\tau+Nh}^{-k\tau}-\hat{X}_{-k\tau+(N-1)h}^{-k\tau},\hat{X}_{-k\tau+Nh}^{-k\tau} \rangle.
\end{split}
\end{align}
From \eqref{eq:RandM} we have that
\begin{align}\label{eq:bound_eq}
\begin{split}
      & 2\langle \hat{X}_{-k\tau+Nh}^{-k\tau}-\hat{X}_{-k\tau+(N-1)h}^{-k\tau},\hat{X}_{-k\tau+Nh}^{-k\tau} \rangle\\
   &=-2h\langle A\hat{X}_{-k\tau+Nh}^{-k\tau},\hat{X}_{-k\tau+Nh}^{-k\tau} \rangle+2h\langle f\big((j+1)h,\hat{X}_{-k\tau+Nh}^{-k\tau}\big),\hat{X}_{-k\tau+Nh}^{-k\tau}
   \rangle\\
   &\ +2\langle g\big((N-1)h\big)\Delta W_{-k\tau+(N-1)h},\hat{X}_{-k\tau+Nh}^{-k\tau}\rangle. 
\end{split}
\end{align}
Note that $\mathbb{E}\langle g\big((N-1)h\big)\Delta W_{-k\tau+(N-1)h},\hat{X}_{-k\tau+(N-1)h}^{-k\tau}\rangle=0$. Taking the expectation of both sides of \eqref{eq:bound_eq} and making use of Assumption \ref{as:f} give
\begin{align*}
          &\|\hat{X}_{-k\tau+Nh}^{-k\tau}\|^2-\|\hat{X}_{-k\tau+(N-1)h}^{-k\tau}\|^2+\|\hat{X}_{-k\tau+Nh}^{-k\tau}-\hat{X}_{-k\tau+(N-1)h}^{-k\tau}\|^2\\
          &=2\mathbb{E}\langle \hat{X}_{-k\tau+Nh}^{-k\tau}-\hat{X}_{-k\tau+(N-1)h}^{-k\tau},\hat{X}_{-k\tau+Nh}^{-k\tau} \rangle\\
   &\leq -2h\mathbb{E} \langle (A-C_fI)\hat{X}_{-k\tau+Nh}^{-k\tau},\hat{X}_{-k\tau+Nh}^{-k\tau} \rangle+h(2C_f+\sigma^2)+\|\hat{X}_{-k\tau+Nh}^{-k\tau}-\hat{X}_{-k\tau+(N-1)h}^{-k\tau}\|^2. 
\end{align*}
Then cancelling the same term on both side gives
\begin{align*}
    &\|\hat{X}_{-k\tau+Nh}^{-k\tau}\|^2-\|\hat{X}_{-k\tau+(N-1)h}^{-k\tau}\|^2\\
    &\leq -2h\mathbb{E} \langle (A-C_fI)\hat{X}_{-k\tau+Nh}^{-k\tau},\hat{X}_{-k\tau+Nh}^{-k\tau} \rangle+h(2C_f+\sigma^2)\\
    &\leq -2h (\lambda_1-C_f)\|\hat{X}_{-k\tau+Nh}^{-k\tau}\|^2+h(2C_f+\sigma^2).
\end{align*}
Let $\alpha:=\frac{2C_f+\sigma^2}{2(\lambda_1-C_f)}$. Rearranging the terms above gives
\begin{align}
    \big(1+2h(\lambda_1-C_f)\big)\big(\|\hat{X}_{-k\tau+Nh}^{-k\tau}\|^2-\alpha\big)\le \|\hat{X}_{-k\tau+(N-1)h}^{-k\tau}\|^2-\alpha.
\end{align}
By iteration, this leads to
\begin{align}
    \|\hat{X}_{-k\tau+Nh}^{-k\tau}\|^2\leq \frac{1}{\big(1+2h(\lambda_1-C_f)\big)^N}\big(\|\xi\|^2-\alpha\big)+\alpha.
\end{align}
Because of Assumption \ref{as:f_constant} and \ref{as:ini}, the term on the right hand side above can be bounded by $\|\xi\|^2+\alpha$, which is independent of $k$, $N$ and $h$.
\end{proof}
The next result shows two numerical solutions starting from different initial conditions can be arbitrarily close after sufficiently many iterations.
\begin{lemma}\label{lem:sta1}
Under Assumption \ref{as:A} to \ref{as:ini}, define $\hat{X}_{-k\tau+Nh}^{-k\tau}$ and $\hat{Y}_{-k\tau+Nh}^{-k\tau}$ solutions of the backward Euler-Maruyama scheme on $\mathcal{T}^h$. Then there exists an $N^*$ such that for any $N\geq N^*$, $\|\hat{X}_{-k\tau+Nh}^{-k\tau}-\hat{Y}_{-k\tau+Nh}^{-k\tau}\|<\epsilon$.
\end{lemma}
\begin{proof}
Define $D_N:=\hat{X}_{-k\tau+Nh}^{-k\tau}-\hat{Y}_{-k\tau+Nh}^{-k\tau}$. Let us use \eqref{eqn:eqa} again, which allows us to examine the following term:
\begin{align*}
   2\mathbb{E} \langle D_N-D_{N-1} ,D_N\rangle&=-2h\mathbb{E}\langle AD_N,D_N\rangle\\
   &+2h\mathbb{E}\langle f\big((j+1)h,\hat{X}_{-k\tau+Nh}^{-k\tau}\big)-f\big((j+1)h,\hat{Y}_{-k\tau+Nh}^{-k\tau}\big),D_N
   \rangle\\
   &\leq 2h\mathbb{E}\langle(-A+C_fI)D_N,D_N\rangle.
\end{align*}
This leads to
\begin{align*}
    (1+2h(\lambda_1-C_f))\|D_N\|\leq \|D_{N-1}\|.
\end{align*}
By iteration we have
\begin{align*}
    \|D_N\|\leq \frac{1}{(1+2h(\lambda_1-C_f))^N}\|D_{0}\|=\frac{1}{(1+2h(\lambda_1-C_f))^N}\|\xi-\eta\|.
\end{align*}
Because of $\lambda_1>C_f$, the assertion follows.
\end{proof}
\begin{proof}[Proof of Theorem \ref{thm:main2}]
First we shall show that there exists a limit of $\hat{X}_{0}^{-k\tau}$ in $L^2(\Omega)$. Note from Lemma \ref{lem:boundedness}, it holds $\hat{X}_{-k\tau+Nh}^{-k\tau}\in L^2(\Omega)$ for $N\in\mathbb{N}$.  For $t=-k\tau+Nh$, by using the semi-flow property we have for $m\in\mathbb{N}$
\begin{align*}
    \hat{X}^{-k\tau-m\tau}_t=\hat{X}^{-k\tau}_t\circ \hat{X}^{-k\tau-m\tau}_{-k\tau}.
\end{align*}
Both sides are the same process and $\hat{X}^{-k\tau}_t$ on the RHS has a different initial condition.
Denote $M:=nk$, then by Lemma \ref{lem:sta1} we have for $\epsilon>0$ there exists a $M^*$ such that for $M\geq M^*$
\begin{align*}
    \big\|\hat{X}^{-k\tau-m\tau}_t-\hat{X}^{-k\tau}_t\big\|=\big\|\hat{X}^{-(M+nm)h}_t-\hat{X}^{-Mh}_t\big\|<\epsilon.
\end{align*}
Then we construct the Cauchy sequence $(\hat{X}^{-k\tau}_t)_{k\in \mathbb{N}}$ converging to some limit $\hat{X}^*$ in $L^2(\Omega)$. Also it is not hard to show that the convergence is independent of the initial point. For $k\to\infty$, we have from  Lemma \ref{lem:sta1}
\begin{align*}
   \big \|\hat{X}^*-\hat{X}^{-k\tau}_t(\eta)\big\|\leq  \big \|\hat{X}^*-\hat{X}^{-k\tau}_t(\xi)\big\|+ \big \|\hat{X}^{-k\tau}_t(\xi)-\hat{X}^{-k\tau}_t(\eta)\big\|\to 0.
\end{align*}
Now let us verify the random periodicity of the backward Euler-Maruyama scheme by induction. Let us examine two terms $\hat{X}_{-k\tau+Nh}^{-k\tau}(\theta_\tau \omega)$ and $\hat{X}_{-(k-1)\tau+Nh}^{-(k-1)\tau}(\omega)$, where $t=-k\tau+Nh$. For $\hat{X}_{-k\tau+Nh}^{-k\tau}(\theta_\tau \omega)$ we have the expression
\begin{align*}
    \hat{X}_{-k\tau+Nh}^{-k\tau}(\theta_\tau\omega)& = \hat{X}_{-k\tau+(N-1)h}^{-k\tau}(\theta_\tau\omega)-Ah\hat{X}_{-k\tau+Nh}^{-k\tau}(\theta_\tau\omega) \\
   &+h f\big(Nh,   \bar{X}_{-k\tau+Nh}^{-k\tau}(\theta_\tau\omega) \big)+ 
    g(jh)\Delta W_{-k\tau+(N-1)h}(\theta_\tau\omega),
\end{align*}
where 
$$\Delta W_{-k\tau+(N-1)h}(\theta_\tau\omega)=W_{-(k-1)\tau+Nh}-W_{-(k-1)\tau+(N-1)h}=\Delta W_{-(k-1)\tau+(N-1)h}(\omega).$$
For $\hat{X}_{-(k-1)\tau+Nh}^{-(k-1)\tau}(\omega)$, we have its expression given by
\begin{align*}
     \hat{X}_{-(k-1)\tau+Nh}^{-(k-1)\tau}(\omega) =& \hat{X}_{-(k-1)\tau+(N-1)h}^{-(k-1)\tau}(\omega)-Ah\hat{X}_{-(k-1)\tau+Nh}^{-(k-1)\tau}(\omega)\\
    & f( Nh,   \bar{X}_{-(k-1)\tau+Nh}^{-(k-1)\tau} ( \omega)\big)+
    g(jh)\Delta W_{-(k-1)\tau+(N-1)h}(\omega).
\end{align*}
By induction and by the pathwise uniqueness of the solution of the backward Euler-Maruyama scheme (Theorem \ref{thm:well-posedness}), we have that
\begin{align*}
   \hat{X}_{-k\tau+Nh}^{-k\tau}\big(\theta_\tau \omega,\xi(\theta_\tau \omega)\big) =\theta_\tau\hat{X}_{-k\tau+Nh}^{-k\tau}\big(\omega,\xi(\omega)\big) = \hat{X}_{-(k-1)\tau+Nh}^{-(k-1)\tau}\big( \omega,\xi(\omega)\big). 
\end{align*}
Finally from \eqref{eq:limit} and the fact $t=-k\tau+Nh$, we have
\begin{align*}
   & \big\| \hat{X}_{t}^*(\theta_\tau \omega)-\hat{X}_{t+\tau}^{*}( \omega)\big\|\\
   &\leq  \big\| X_{t}^{-k\tau}\big(\theta_\tau \omega,\xi(\theta_\tau \omega)\big)-\hat{X}_{t}^*(\theta_\tau \omega)\big\|+\big\| \hat{X}_{t+\tau}^{-(k-1)\tau}\big( \omega,\xi(\omega)\big)-\hat{X}_{t+\tau}^{*}( \omega)\big\|\overset{k\to \infty}{\longrightarrow}0.
\end{align*}
Therefore,  $\hat{X}_{t}^*(\theta_\tau \omega)=\hat{X}_{t+\tau}^{*}( \omega)$ $\mathbb{P}$-a.s.
\end{proof}
\section{Error analysis}
\begin{thm}\label{thm:error} Under Assumption \ref{as:A} to \ref{as:ini} and Assumption \ref{as:f_Khasminskii-type}, for any $h\in(0,1)$ with $\tau=nh$, $n\in\mathbb{N}$, there exists a constant $C$ that depends on $q,A,f,g$ and $d$ such that the backward Euler-Maruyama method \eqref{eq:RandM} approximates the true solution of \eqref{eq:SPDE} on $\mathcal{T}^h$ with  
\begin{align}\label{eq:error1}
 \sup_{k,N}\big\|X^{-k\tau}_{-k\tau+Nh}-\hat{X}^{-k\tau}_{-k\tau+Nh}\big\|\leq C h^{1/2}.
\end{align}
\end{thm}
\begin{proof}
First note that
\begin{align}
\begin{split}
        X^{-k\tau}_{-k\tau+Nh}&=X^{-k\tau}_{-k\tau+(N-1)h}-\int_{-k\tau+(N-1)h}^{-k\tau+Nh}AX^{-k\tau}_{s}\mathrm{d}s\\
    &+\int_{-k\tau+(N-1)h}^{-k\tau+Nh}f\big(s,X^{-k\tau}_{s}\big)\mathrm{d}s+\int_{-k\tau+(N-1)h}^{-k\tau+Nh}g(s)\mathrm{d}W_s\\
    &=X^{-k\tau}_{-k\tau+(N-1)h}-\int_{-k\tau+(N-1)h}^{-k\tau+Nh}A\big(X^{-k\tau}_{s}-X^{-k\tau}_{-k\tau+Nh}\big)\mathrm{d}s-hAX^{-k\tau}_{-k\tau+Nh}\\
    &+\int_{-k\tau+(N-1)h}^{-k\tau+Nh}\Big(f\big(s,X^{-k\tau}_{s}\big)-f\big(s,X^{-k\tau}_{-k\tau+Nh}\big)\Big)\mathrm{d}s+hf\big(s,X^{-k\tau}_{-k\tau+Nh}\big)\\
    &+\int_{-k\tau+(N-1)h}^{-k\tau+Nh}\Big(g(s)-g\big((N-1)h\big)\Big)\mathrm{d}W_s+g\big((N-1)h\big)\Delta W_{-k\tau+(N-1)h}.
\end{split}
\end{align}
Define $e_N:=X^{-k\tau}_{-k\tau+Nh}-\hat{X}^{-k\tau}_{-k\tau+Nh}$. Then 
\begin{align*}
   &2\mathbb{E} \langle e_N-e_{N-1} ,e_N\rangle\\
   &=-2h\mathbb{E}\langle Ae_N,e_N\rangle+2h\mathbb{E}\langle f\big((j+1)h,X_{-k\tau+Nh}^{-k\tau}\big)-f\big((j+1)h,\hat{X}_{-k\tau+Nh}^{-k\tau}\big),e_N
   \rangle\\
   &+2\mathbb{E}\Big\langle-\int_{-k\tau+(N-1)h}^{-k\tau+Nh}A\big(X^{-k\tau}_{s}-X^{-k\tau}_{-k\tau+Nh}\big)\mathrm{d}s ,e_N\Big\rangle\\
   &+2\mathbb{E}\Big\langle\int_{-k\tau+(N-1)h}^{-k\tau+Nh}\Big(f\big(s,X^{-k\tau}_{s}\big)-f\big(s,X^{-k\tau}_{-k\tau+Nh}\big)\Big)\mathrm{d}s ,e_N\Big\rangle\\
   &+2\mathbb{E}\Big\langle\int_{-k\tau+(N-1)h}^{-k\tau+Nh}\Big(g(s)-g\big((N-1)h\big)\Big)\mathrm{d}W_s ,e_N\Big\rangle.
  \end{align*}
By the Young's inequality 
$$2ab\leq \epsilon^2a^2+\frac{b^2}{\epsilon^2}, \forall a,b>0,$$
and Assumption \ref{as:f}, we are able to choose $\epsilon_0^2:=h(\lambda_1-C_f)$ such that
\begin{align*}
&2\mathbb{E} \langle e_N-e_{N-1} ,e_N\rangle\\
     &\leq 2h\mathbb{E}\langle(-A+C_fI)e_N,e_N\rangle+3\epsilon_0^2 \|e_N\|^2\\
   &+\frac{1}{\epsilon^2_0}\Big\|-\int_{-k\tau+(N-1)h}^{-k\tau+Nh}A\big(X^{-k\tau}_{s}-X^{-k\tau}_{-k\tau+Nh}\big)\mathrm{d}s\Big\|^2\\
   &+\frac{1}{\epsilon^2_0}\Big\|\int_{-k\tau+(N-1)h}^{-k\tau+Nh}\Big(f\big(s,X^{-k\tau}_{s}\big)-f\big(s,X^{-k\tau}_{-k\tau+Nh}\big)\Big)\mathrm{d}s\Big\|^2\\
   &+\frac{1}{\epsilon^2_0}\Big\|\int_{-k\tau+(N-1)h}^{-k\tau+Nh}\Big(g(s)-g\big((N-1)h\big)\Big)\mathrm{d}W_s\Big\|^2.
\end{align*}
By Proposition \ref{prop:use_bound}, we know there exists a constant $C$ depending on $q$, $A$, $f$ and $g$ such that
\begin{align*}
    &\Big\|-\int_{-k\tau+(N-1)h}^{-k\tau+Nh}A\big(X^{-k\tau}_{s}-X^{-k\tau}_{-k\tau+Nh}\big)\mathrm{d}s\Big\|^2\\
   &+\Big\|\int_{-k\tau+(N-1)h}^{-k\tau+Nh}\Big(f\big(s,X^{-k\tau}_{s}\big)-f\big(s,X^{-k\tau}_{-k\tau+Nh}\big)\Big)\mathrm{d}s\Big\|^2\\
   &+\Big\|\int_{-k\tau+(N-1)h}^{-k\tau+Nh}\Big(g(s)-g\big((N-1)h\big)\Big)\mathrm{d}W_s\Big\|^2\\
   &\leq Ch^3\Big(1+\sup_{k,N} \|X^{-k\tau}_{-k\tau+Nh}\|_{4q-2}^{2q-1}\Big):=\beta h^3.
\end{align*}
Note that $\beta$ is bounded because of Proposition \ref{prop:use_bound}.
Then from \eqref{eqn:eqa} and the estimate above we have that
\begin{align*}
      \|e_N\|^2-\|e_{N-1}\|^2\leq 2\mathbb{E} \langle e_N-e_{N-1} ,e_N\rangle
   &\leq 2h\mathbb{E}\langle(-A+C_fI)e_N,e_N\rangle+3\epsilon_0^2 \|e_N\|^2+\frac{\beta h^3}{\epsilon^2_0}.
\end{align*}
Define $\hat{\alpha}:=\frac{\beta h}{(\lambda_1-C_f)^2}$. The inequality above can be rearranged to
\begin{align*}
    \Big(1+h(\lambda_1-C_f)\Big)\big(\|e_N\|^2-\hat{\alpha}\big)\leq \|e_{N-1}\|^2-\hat{\alpha}.
\end{align*}

By iteration and assuming $\hat{X}^{-k\tau}_{-k\tau}=X^{-k\tau}_{-k\tau}=\xi$ we have 
$$\|e_N\|^2\leq \Big(1-\frac{1}{1+h(\lambda_1-C_f)^N}\Big)\frac{\beta h}{(\lambda_1-C_f)^2}$$
Finally due to Assumption \ref{as:f_constant} (alternatively, Assumption \ref{as:f_Khasminskii-type}), we have $\|e_N\|^2\leq \frac{\beta h}{(\lambda_1-C_f)^2}.$ Then the assertion follows.
\end{proof}
\begin{corr}\label{corr:error} Under Assumption \ref{as:A} to \ref{as:f_tangent} and Assumption \ref{as:f_Khasminskii-type}, for any $h\in(0,1)$ with $\tau=nh$, $n\in\mathbb{N}$, there exists a constant $C$ that depends on $q,A,f,g$ and $d$ such that the exact and numerical random periodic solutions of \eqref{eq:RandM} given in Theorems \ref{thm:main1} and \ref{thm:main2} satisfy
\begin{align}\label{eq:error}
    \sup_{t\in \mathcal{T}^h} \big\|X^*_t-\hat{X}^*_t\big\|\leq C h^{\frac{1}{2}}.
\end{align}
\end{corr}
\begin{proof}
The result simply follows from
\begin{align*}
   \big\|X^*_t-\hat{X}^*_t\big\|\leq \limsup_{k}\Big[\big\|X^*_t-\hat{X}^{-k\tau}_t\big\|+\big\|X^{-k\tau}_t-\hat{X}^{-k\tau}_t\big\|+\big\|\hat{X}^{-k\tau}_t-\hat{X}^*_t\big\|\Big]. 
\end{align*}
\end{proof}
\subsection{The periodic measure}\label{sec:periodic}
Theorem \ref{thm:main1} ensures the existence and uniqueness of the random periodic solution to Eqn. \eqref{eq:RandM}. Then the existence of the periodic measure $\rho$ associated with the random periodic semiflow generated by Eqn. \eqref{eq:RandM} can follow from the result in \cite{feng2020}. It can be defined as the law of random periodic solutions, ie,
\begin{equation}
    \rho_t(\Gamma)=P(X^*_t \in \Gamma) \ \forall t \in \mathbb{R},
\end{equation}
where $P$ is the transition probability defined in Definition \ref{def:measure}. Following the argument in Section 5 \cite{rpsnumerics2017}, 
the transition probability induces a semi-group defined by 
$$ P(t + s, s)\phi(\xi):=\int_{\mathbb{R}^d}P(t + s, s, \xi, \mathrm{d}\mu)\phi(\mathrm{\mu})=\mathbb{E}\phi\big({X}^{s}_{t+s}(\xi)\big),$$
where $\phi$ is bounded and measurable.

Similarly, for a fixed $h$, we can define the transition probability of the discrete semi-flow $\hat{u}$ from the backward Euler–Maruyama
scheme by
\begin{equation}
    \hat{P}(t + s, s, \xi, \Gamma):=\hat{P}(\{\omega:u(t+s,s,\omega)\xi\in \Gamma\})=\hat{P}(\hat{X}^{s}_{t+s}\in \Gamma)
\end{equation}
for any $s=jh$ and $t=ih$, $j\in \mathbb{Z}$, $i\in \mathbb{N}$. This newly defined transition probability also induces a semi-group defined by 
$$ \hat{P}(t + s, s)\phi(\xi):=\int_{\mathbb{R}^d}\hat{P}(t + s, s, \xi, \mathrm{d}\mu)\phi(\mathrm{\mu})=\mathbb{E}\phi\big({\hat{X}}^{s}_{t+s}(\xi)\big).$$
Using the result in \cite{feng2020} again yields that the measure function defined by 
$$\hat{\rho}_{s}(\Gamma):=\hat{P}(X^*_s \in \Gamma)\ \text{for }s\in \mathcal{T}^h$$ is periodic. This implies that
for any $s\in \mathcal{T}^h$ and $t=jh,\ j\in \mathbb{Z}$, we have
\begin{equation}
        \hat{\rho}_{s+\tau}=\hat{\rho}_{s},\ \ \int_{\mathbb{R}^d}\hat{P}(t + s, s, x, \Gamma)\hat{\rho}_s(\mathrm{d}x)=\hat{\rho}_{t+s}(\Gamma).
\end{equation}

\begin{thm}\label{thm:periodic}
Under Assumption \ref{as:A} to \ref{as:f_tangent} and Assumption \ref{as:f_Khasminskii-type}. Let $h\in (0,1)$ with $\tau=nh$, $n\in \mathbb{N}$. Then periodic measures $\rho_{\cdot}$ and $\hat{\rho}_{\cdot}$ generated by the exact solution of Eqn. \eqref{eq:SPDE} and the numerical approximation \eqref{eq:RandM} are weak limits of transition
probabilities, ie,
\begin{equation}
    P(r,-k\tau,\xi)\to \rho_{r}, \ \ \hat{P}(t,-k\tau,\xi)\to \hat{\rho}_{t}
\end{equation}
as $k\to \infty$ weakly, where $r\in \mathbb{R}$ and $t\in \mathcal{T}^h$. Moreover, there exists a constant $K$ depending on $q,A,f,g$ and $d$ such that
\begin{equation}
    \sup_{\mathcal{S}}\Big|\int_{\mathbb{R}^d}\phi(x)\rho_t(\mathrm{d}x)-\int_{\mathbb{R}^d}\phi(x)\hat{\rho}_t(\mathrm{d}x)\Big|\leq Kh^{\frac{1}{2}},
\end{equation}
on $\mathcal{T}^h$, where $\mathcal{S}:=\{\phi: \mathbb{R}^d\to \mathbb{R}, |\phi(x)-\phi(x)|\leq |x-y| \text{ and } |\phi(\cdot)|\leq 1\}.$
\end{thm}

With Theorem \ref{thm:main1}, Theorem \ref{thm:main2} and Theorem \ref{thm:error}, the proof simply follows a similar argument as in the proof of Theorem 5.2 \cite{rpsnumerics2017}.


\bibliographystyle{unsrt}  

\section{Numerical analysis}
In this section, we consider the following one-dimensional SDE example
\begin{align}\label{eqn:sde_eg}
\mathrm{d}X_t^{t_0}=-10\pi X_t^{t_0}\mathrm{d}t+\sin{(2\pi t)}\mathrm{d}t+0.05\mathrm{d}W_t. 
\end{align}
It is easily verified that the associated period is $1$ and Assumption \ref{as:A} to \ref{as:f_tangent} and Assumption \ref{as:f_Khasminskii-type} are fulfilled with $\lambda_1=10$, $C_f=2$ and $\sigma=0.05$. Thus \eqref{eqn:sde_eg} has a random periodic solution according to Theorem \ref{thm:main1} and its backward Euler–Maruyama simulation also admit a random periodic path. First, let us show, the scheme converges to its random periodic path regardless its initial condition. To achieve this, we choose the time grid between $t_0=-10$ and $T=0$ with stepsize $0.05$, generate a Brownian realisation on the time grid, and set two initial conditions to be $0.2$ and $-0.3$. Two simulated paths can then be obtained in Figure \ref{fig:initial} by applying the backward Euler–Maruyama method in \eqref{eq:RandM} iteratively on the time grid, with given initial condition and shared Brownian realisation. As shown in Figure \ref{fig:initial}, two paths coincide shortly after the start. Note in theory $\hat{X}_{t}^{*}=\hat{X}_{t}^{-\infty}$, but we take pull-back time $-10$ as this is already enough to generate a good convergence to the random periodic paths for $t\geq -9$.
\begin{figure}
\centering
\includegraphics[width=0.65\textwidth]{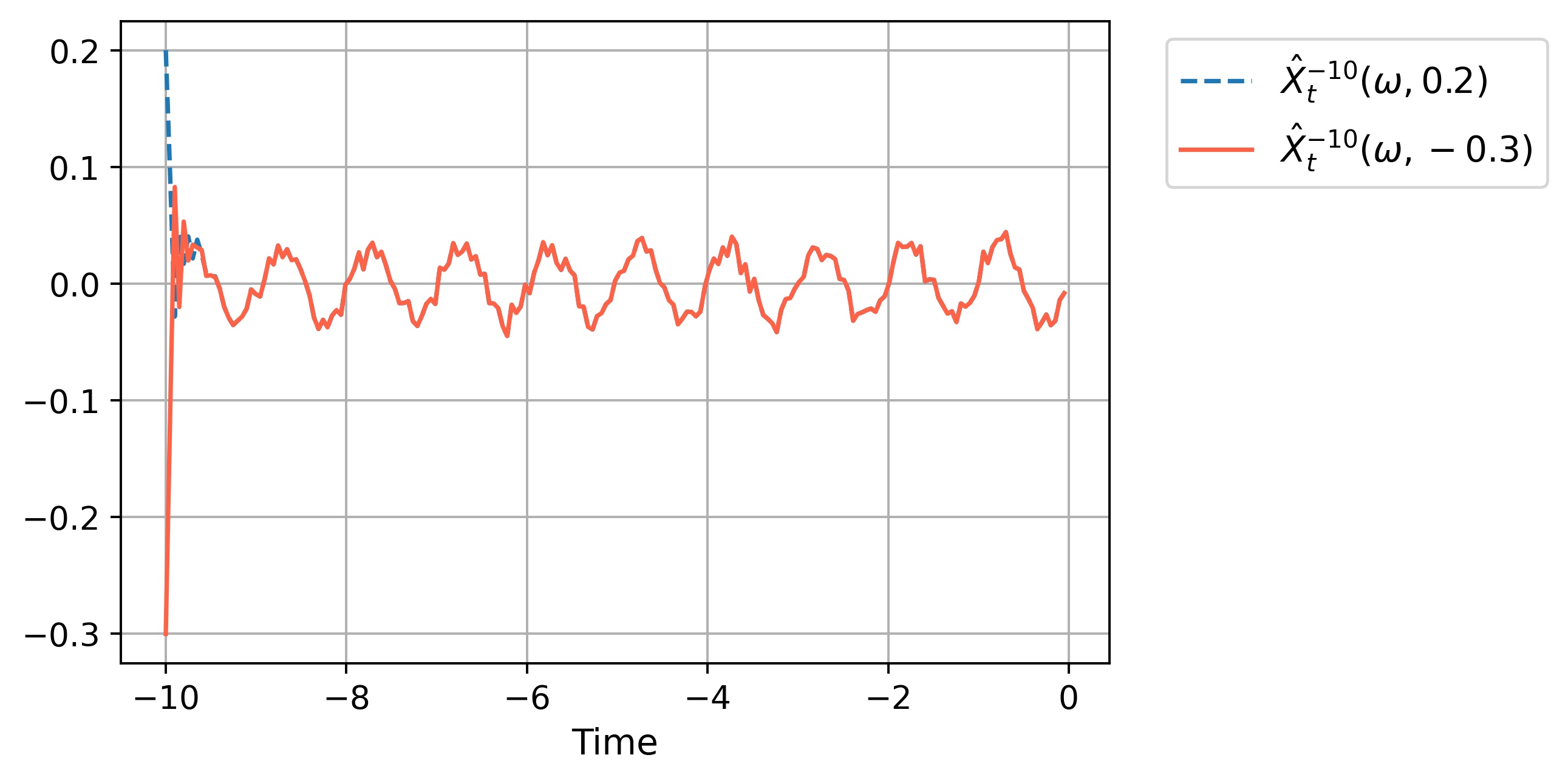}
\caption{Two paths generated by backward Euler–Maruyama method from differential initial conditions. \label{fig:initial}
}
\end{figure}

As discussed in \cite{rpsnumerics2017}, there are two ways to demonstrate the periodicity. The easer approach is to simulate the processes $\hat{X}_{t}^{*}(\omega)=X^{-30}_t(\omega,0.2)$for $t\in [-4,-1]$ and $\hat{X}_{t}^{*}(\theta_{-1}\omega)=X^{-30}_t(\theta_{-1}\omega,0.2)$ for $t\in [-3,0]$. We can observe that the two segmented processes are identical in Figure \ref{fig:firstmethod} due to $\hat{X}_{t-1}^{*}(\omega)=\hat{X}_{t}^{*}(\theta_{-1}\omega)$.

\begin{figure}
\centering
\includegraphics[width=0.65\textwidth]{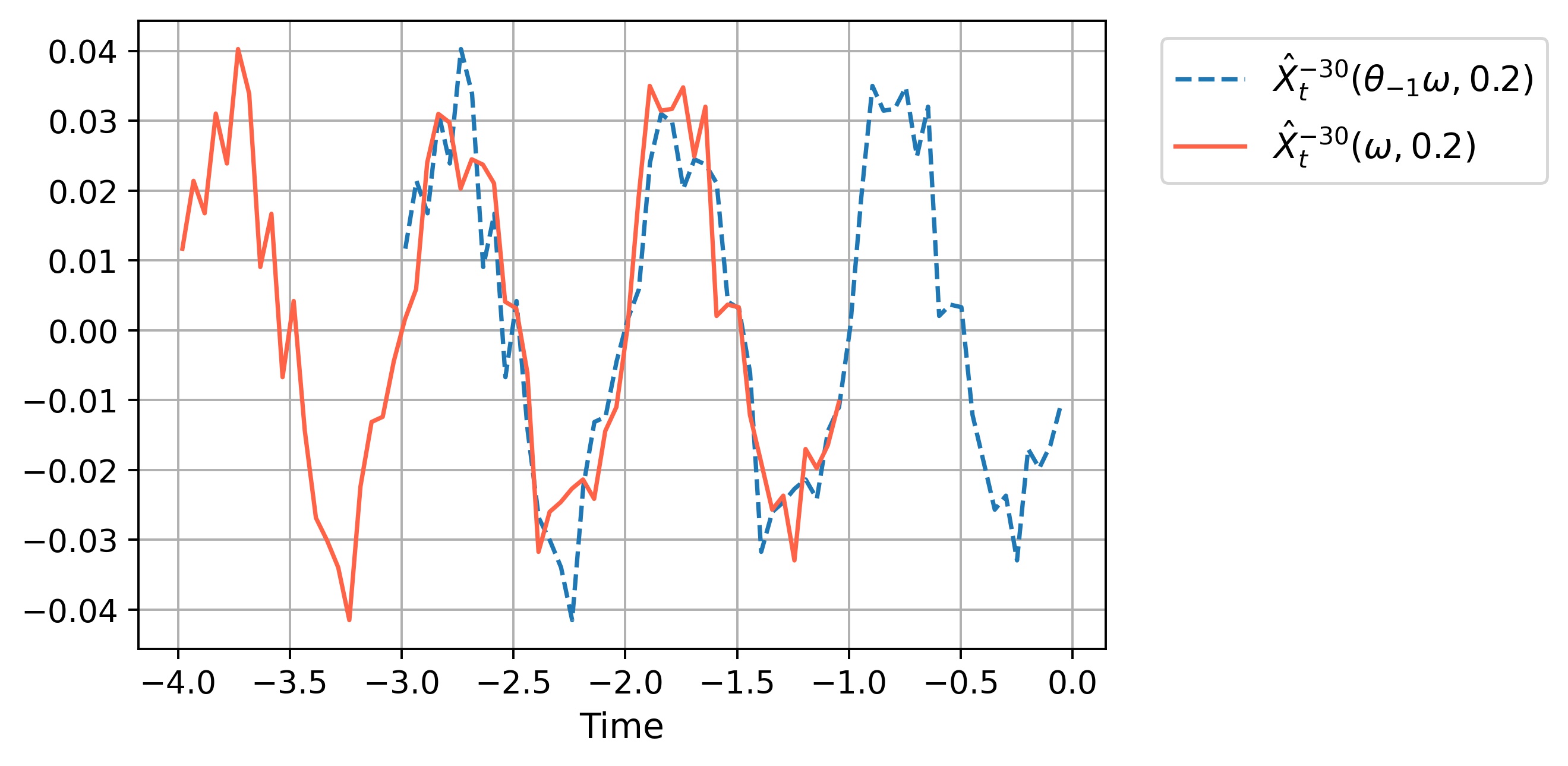}
\caption{Two paths with generated by backward Euler–Maruyama method on different realisations. \label{fig:firstmethod}
}
\end{figure}

\begin{figure}
\centering
\includegraphics[width=0.5\textwidth]{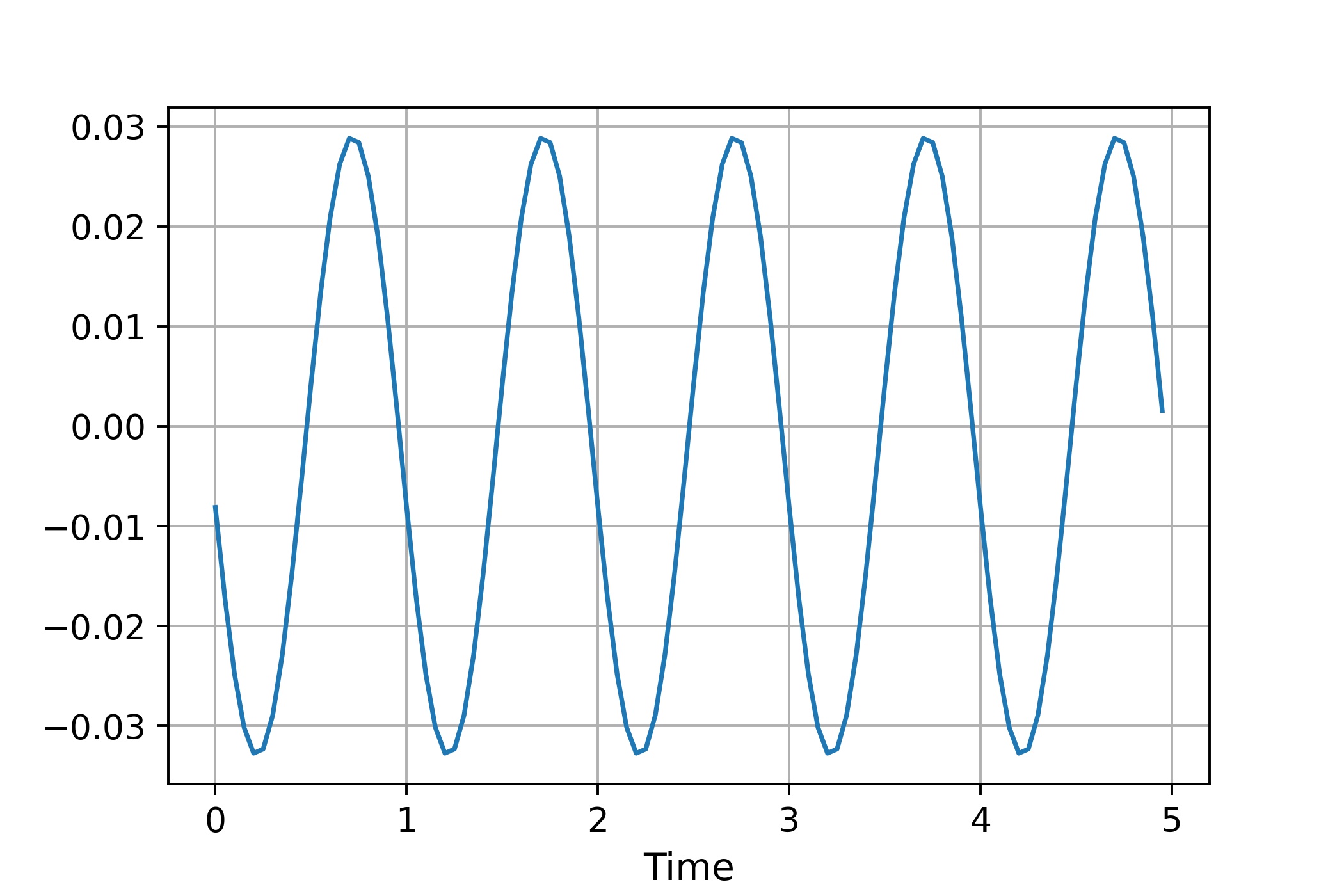}
\caption{The pull-back path $\hat{X}^{-30}(t,\theta_{-t}\omega)$ generated by backward Euler–Maruyama method. \label{fig:secondmethod}
}
\end{figure}

The other way to check random periodicity of path $X$ with period $\tau$ is to verify whether or not $\hat{X}^{*}(t,\theta_{-t}\omega)$ is periodic with period $\tau$. 
To test it, we need to consider $X_t^{t_0}(\theta_{-t}\omega)$. Note that for any fixed $r\in \mathbb{R}$ we have that
\begin{align}\label{eqn:pullbackpath}
\begin{split}
\mathrm{d}X_t^{t_0}(\theta_{-r}\omega)&=-10\pi X_t^{t_0}(\theta_{-r}\omega)\mathrm{d}t+\sin{(2\pi t)}\mathrm{d}t+0.05\mathrm{d}W_{t}(\theta_{-r}\omega)\\
&=-10\pi X_t^{t_0}(\theta_{-r}\omega)\mathrm{d}t+\sin{(2\pi t)}\mathrm{d}t+0.05\mathrm{d}W_{t-r}.
\end{split}
\end{align}
Now set $t_0=0$, and $X_{0}^{0}(\theta_{-r}\omega)=x_0$. For each fixed $r$, we simulate the path of Eqn. \eqref{eqn:pullbackpath} through the backward Euler-Maruyama method up to $t=r$. Then we obtain the evaluation of $\hat{X}^{0}(r,\theta_{-r}\omega)$. To allow convergence, we look at the path pattern from $t=2$ to $t=5$ in Figure \ref{fig:secondmethod}. Apparently we have obtained a periodic pull-back path as expected, which in turn shows the random periodicity of the original path.

Finally, we test the order of convergence of the backward Euler-Maruyama method and compare the performance with (forward) Euler-Maruyama method. For its approximation we first generated
a reference solution with a small step size of $h_{\text{ref}}=2^{-15}$. This reference solution was
then compared to numerical solutions with larger step sizes $h\in \{2^{-i}: i=4,5,6,7,8\}$. The error plot is shown in Figure \ref{fig:errorplot}. We plot the Monte Carlo
estimates of the root-mean-squared errors versus the underlying temporal step size,
i.e., the number i on the x-axis indicates the corresponding simulation is based on
the temporal step size $h = 2^{-i}$. Both methods give the order of convergence above $1$, which is beyond the theoretical order of convergence. When the stepsize is large, say, $h=2^{-4}$, the Euler-Maruyama method has the error 0.048 which is almost five times of the error 0.011 from the backward Euler-Maruyama method. Indeed if we relax the stepsize to $h=2^{-3}$, the Euler-Maruyama diverges while the backward Euler-Maruyama method still converges as expected. This further supports Theorem \ref{thm:main2} and the advantage of backward Euler-Maruyama method: the backward Euler-Maruyama method converges regardless the size of stepsize ($h<1$).

\begin{figure}
\centering
\includegraphics[width=0.5\textwidth]{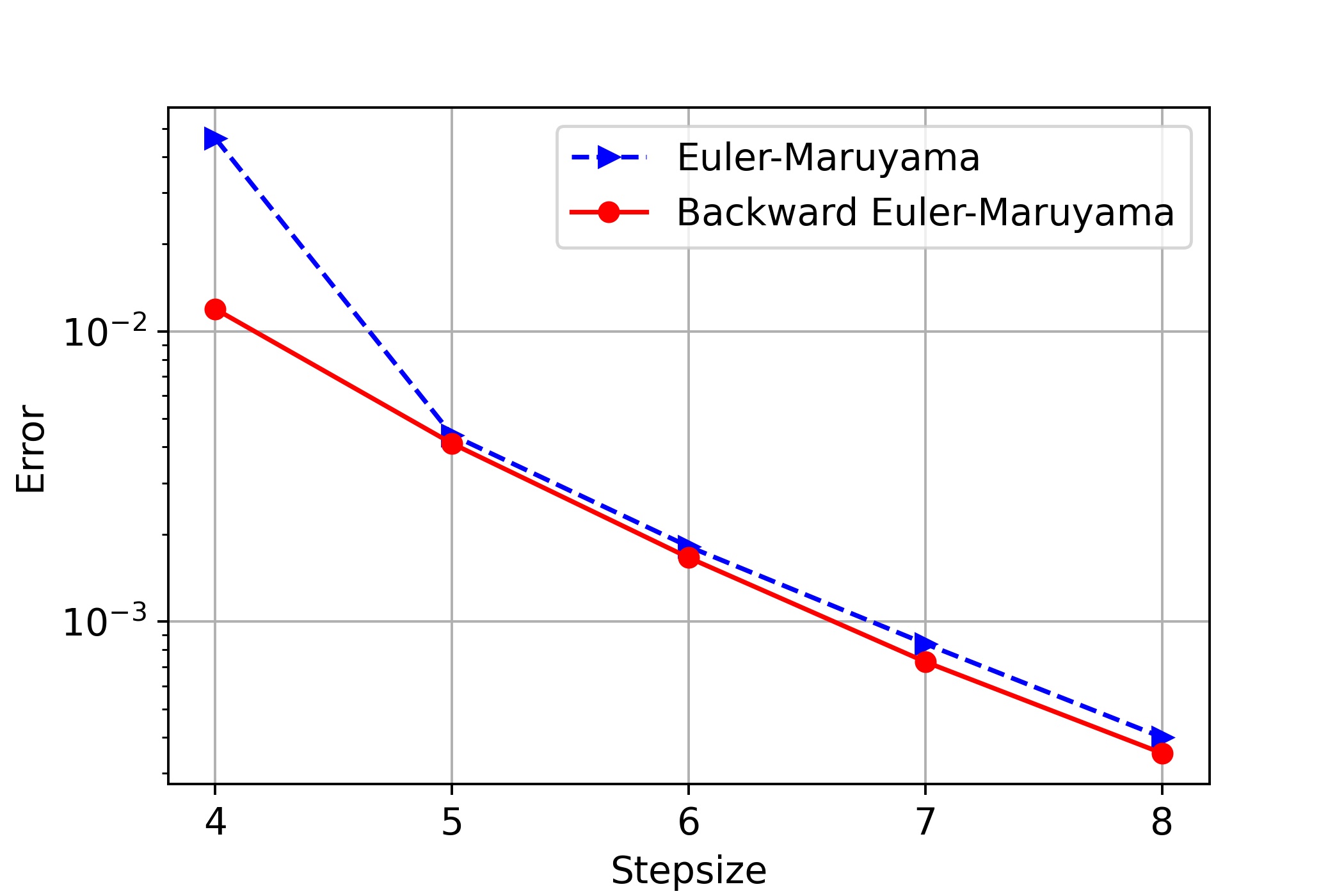}
\caption{ Numerical experiment for simulating the random periodic solution of SDE \eqref{eqn:sde_eg}: Step sizes versus $L^2$ error. \label{fig:errorplot}
}
\end{figure}
\section*{Acknowledgements}
This work is supported by the Alan Turing Institute for funding this work under EPSRC grant EP/N510129/1 and EPSRC for funding though the project EP/S026347/1, titled 'Unparameterised multi-modal data, high order signatures, and the mathematics of data science'. The author would also acknowledge Michael Scheutzow for useful discussion.

\bibliographystyle{plain}

\bibliography{bibfile}

\end{document}